\documentclass[11pt]{amsart}
\usepackage{amssymb,amsmath,latexsym,enumerate, dsfont}
\usepackage{graphicx,verbatim,enumerate,%
ifpdf,mdwlist, xspace}
\usepackage{color}
\usepackage{mathabx}
\ifpdf
\usepackage{hyperref}%[unicode]{hyperref}
\else
\usepackage[hypertex]{hyperref}
\fi
\usepackage{graphicx}
\usepackage{cancel}
\usepackage{float}
\renewcommand{\phi}{\varphi}

\newtheorem{theorem}{Theorem}[section]

\newtheorem{lemma}[theorem]{Lemma}
\newtheorem{corollary}[theorem]{Corollary}

\newtheorem{observation}[theorem]{Observation}
\newtheorem{defi}[theorem]{Definition}
\newenvironment{emdef}{\begin{defi} \rm}{ \end{defi}}
\newtheorem{exa}[theorem]{Example}

\newenvironment{remark}{\begin{rem} \rm}{ \end{rem}}

\newtheorem{rem}[theorem]{Remark}
\newtheorem{claim}[theorem]{Claim}

\DeclareMathOperator{\range}{range}

\DeclareMathOperator{\Id}{Id}

\DeclareMathOperator{\I}{\mathcal{I}}

\DeclareMathOperator{\SF}{SF}

\DeclareMathOperator{\diagmod}{DiagonalizationModule}
\DeclareMathOperator{\emdiagmod}{\textit{DiagonalizationModule}}

\newcommand{\rel}[1]{\mathrel{#1}}

\newcommand{\dmpar}[1]{\diagmod{(#1)}}
\newcommand{\emdmpar}[1]{\emdiagmod{(#1)}}

\newcommand{\concat}{\widehat{\phantom{\alpha}}}

\newcommand{\HSF}{hereditarily self-full\xspace}
\newcommand{\CCR}{co-ceer resistant\xspace}

\DeclareMathOperator{\Ceers}{\mathbf{Ceers}}

\DeclareMathOperator{\dark}{\textrm{Dark}}

\usepackage{pifont}
\def\G1{\hbox{$\displaystyle{\mbox{\ding{172}}}$}}

\title[Self-full ceers and the uniform join operator]{Self-full ceers and the uniform join operator}

\author[Andrews]{Uri Andrews}
\address{Department of Mathematics\\
	University of Wisconsin\\
	Madison, WI 53706-1388\\
	USA}
\email{\href{mailto:andrews@math.wisc.edu}{andrews@math.wisc.edu}}
\urladdr{\url{http://www.math.wisc.edu/~andrews/}}

\author[Schweber]{Noah Schweber}
\address{Department of Mathematics\\
	University of Wisconsin\\
	Madison, WI 53706-1388\\
	USA}
\email{\href{mailto:ndschweber@gmail.com}{ndschweber@gmail.com}}

\author[Sorbi]{Andrea Sorbi}
\address{Dipartimento di Ingegneria Informatica e Scienze Matematiche\\
	Universit\`a Degli Studi di Siena\\
	I-53100 Siena, Italy}
\email{\href{mailto:andrea.sorbi@unisi.it}{andrea.sorbi@unisi.it}}
\urladdr{\url{http://www3.diism.unisi.it/~sorbi/}}

\thanks{Andrews was partially supported
by NSF Grant 1600228. Andrews and Sorbi were supported by the project
\emph{Positive preorders and computable reducibility on them, as a
mathematical model of databases}, grant number AP05131579 of the Science
Committee of the Republic of Kazakhstan. Sorbi is a member of
INDAM-GNSAGA. Sorbi's research was partially supported by PRIN 2017 Grant
``Mathematical Logic: models, sets, computability''.}

\keywords{Computably enumerable equivalence relation; self-full
equivalence relation; computable reducibility on equivalence relations}

\subjclass[2010]{03D30, 03D45}

\begin{document}

\begin{abstract}
A computably enumerable equivalence relation (ceer) $X$ is called self-full
if whenever $f$ is a reduction of $X$ to $X$ then the range of $f$
intersects all $X$-equivalence classes. It is known that the infinite
self-full ceers properly contain the dark ceers, i.e. the infinite ceers
which do not admit an infinite computably enumerable transversal. Unlike
the collection of dark ceers, which are closed under the operation of uniform join, we
answer a question from \cite{joinmeet} by showing that there are
self-full ceers $X$ and $Y$ so that their uniform join $X\oplus Y$ is
non-self-full. We then define and examine the  \HSF ceers, which are the
self-full ceers $X$ so that for any self-full $Y$, $X\oplus Y$ is also
self-full: we show that they are closed under uniform join, and that every non-universal degree in $\Ceers_{{}/\I}$ have infinitely many incomparable \HSF strong minimal covers. In particular, every non-universal ceer is bounded by a  \HSF ceer. Thus the \HSF ceers form a properly intermediate class in between the dark ceers and the infinite
self-full ceers which is closed under $\oplus$.
\end{abstract}

\maketitle

\section{Introduction}\label{sct:introduction}
We investigate the notion of \emph{self-fullness}, which has turned out to be
quite important in the study of computably enumerable equivalence relations
(called \emph{ceers}) under computable reducibility. We recall that if $R,S$
are equivalence relations on the set of natural numbers $\omega$, then $R$ is
\emph{computably reducible} (or, simply, \emph{reducible}) to $S$ (notation:
$R \leq S$) if there is a computable total function $f$ such that
\[
(\forall x,y)[x  \rel{R} y \Leftrightarrow f(x) \rel{S} f(y)].
\]
In this case we write $f: R \rightarrow S$. This reducibility (due to
Ershov~\cite{Ershov:NumberingsI}, see also \cite{Ershov:positive}) has been
widely exploited in recent years as a convenient tool for measuring the
computational complexity of classification problems in computable
mathematics. For instance (see~\cite{Fokina-et-al-several}) the isomorphism
relation for various familiar classes of computable groups is
$\Sigma^1_1$-complete under $\leq$. Restricted to ceers, this reducibility
has been used to study familiar equivalence relations from logic, such as the
provable equivalence relation of sufficiently expressive formal systems (see
\cite{andrews2017survey} for a survey), word problems and isomorphism
problems of finitely presented groups (\cite{Miller, Nies-Sorbi}), c.e.\
presentations of structures (see e.g.\
\cite{fokina2016linear,gavruskin2014graphs}).

As a reducibility, $\leq$ gives rise in the usual way to a degree
structure. Due to the importance of ceers within general equivalence
relations on $\omega$, a considerable amount of attention has been given to
its substructure, called $\Ceers$, consisting of the degrees of ceers. The
systematic study of $\Ceers$, a poset with a greatest element (usually called
the \emph{universal element}), was initiated by Gao and Gerdes \cite{Gao-Gerdes}. Its
algebraic structure, in particular its structure under joins and meets was
thoroughly investigated by Andrews and Sorbi~\cite{joinmeet}, who have
proposed a partition of the ceers into the three following classes: the
\emph{finite} ceers, i.e. the ceers with only finitely many equivalence
classes; the \emph{light} ceers, i.e. the ceers $R$ possessing an infinite
c.e. transversal (an infinite c.e. set $W$ such that $x \cancel{\rel{R}} y$
for all $x,y \in W$ with $x \ne y$), or equivalently the ceers $R$ such that
$\Id \leq R$, where $\Id$ denotes the identity ceer); the \emph{dark} ceers,
i.e. the ceers which are neither finite nor light. These classes have been
extensively investigated in relation to the existence or non-existence of
joins and meets in the poset $\Ceers$. For instance no pair of incomparable
degrees of dark ceers has join or meet. The classes of degrees corresponding
to the classes of the above partition are first order definable in $\Ceers$
in the language of posets.

Another class of ceers which has emerged in~\cite{joinmeet} consists of the
self-full ceers, where a ceer $R$ is \emph{self-full} if and only if whenever
a computable function $f$ provides a reduction $f: R\rightarrow R$, then
$\range(f)$ intersects all $R$-equivalence classes. Equivalently, $R$ is
self-full if and only if $R\oplus \Id_1 \nleq R$, where $\Id_{1}$ is the ceer
with exactly one equivalence class. Notice that a degree containing a
self-full ceer consists in fact only of self-full ceers. Let $\mathcal{I}$ ,
$\dark$, $\SF$ denote respectively the classes of finite ceers, dark ceers,
and self-full ceers. It was shown in \cite{joinmeet} that $\mathcal{I} \cup \dark
\subseteq \SF$ and the inclusion is proper. This follows for instance from the fact
(\cite[Theorem~4.10]{joinmeet}) that every non-universal degree in $\Ceers$
is bounded by a self-full degree. Clearly, a self-full ceer bounding a light
ceer cannot be dark. In fact it can be even proved
(\cite[Theorem~7.9]{joinmeet}) that in $\Ceers_{/\I}$ every non-universal element has infinitely many distinct self-full strong
minimal covers. In Theorem \ref{thm:HSF-Istrongminimalcovers}, we show the analogous result for the  \HSF degrees. The partial order $\Ceers_{/\I}$ is defined by
quotienting $\Ceers$ modulo $\equiv_{\I}$, where $R\leq_{\I} S$ if there is
some  $F\in \mathcal{I}$ so that $R \leq S\oplus F$. Recall $\oplus$ is the
operation of \emph{uniform join} on equivalence relations, for which, given
equivalence relations $R,S$, we define $R\oplus S=\{(2x,2y): (x,y) \in R\}
\cup \{(2x+1, 2y+1): (x,y) \in S\}$.

The self-full ceers and their degrees
play an important role in the theory of ceers. The following are a few examples of the prominent role the self-full ceers have played in the theory of ceers. The existence in $\Ceers_{/\I}$ of infinitely
many self-full strong minimal covers above any non-universal degree has been
exploited in the recent paper~\cite{ASSFirstOrderTheory} to show, among other
things, that the first order theory of the degrees of light ceers is
isomorphic to true first-order arithmetic. The self-full degrees form an
automorphism base of the continuum many automorphisms of the poset
$\Ceers$ \cite[Corollary~11.5]{joinmeet}. The degrees of self-full
ceers are first order definable in $\Ceers$ in the language of posets. They coincide exactly with the non-universal meet-irreducible degrees \cite[Theorem 7.8]{joinmeet}.

The three classes of the partition of ceers (finite, light, and dark)
introduced in~\cite{joinmeet} are closed under the operation of uniform join.
Closure under $\oplus$ is an important issue for a class of ceers, not only
as regards the investigation of the existence of infima or suprema of the degrees of
the ceers in the class, but also because uniform joins correspond to
coproducts in the category of equivalence relations. If $R,S$ are equivalence
relations, then a \emph{morphism} $\mu: R \rightarrow S$, from $R$ to $S$, is
a function mapping $R$-equivalence classes to $S$-equivalence classes, for
which there is a computable function $f$ such that for every $x$, $\mu$ maps
the $R$-equivalence class of $x$ to the $S$-equivalence class of $f(x)$. So
a class of ceers is closed under $\oplus$ if and only if the corresponding
full subcategory of ceers is closed under coproducts of any pair of objects.
The category-theoretic approach to numberings and equivalence relations is
due to Ershov, see in particular \cite{Ershov:NumberingsII}. Since reductions
correspond to monomorphisms, by \cite[Lemma~1.1.]{joinmeet} the
category-theoretic jargon allows for yet another characterization of the
self-full ceers,  namely a ceer $R$ is self-full if and only if every
monomorphism $\mu:R \rightarrow R$ is an isomorphism.

A natural question is therefore whether or not the self-full ceers are closed
under $\oplus$ (see \cite[Question~1]{joinmeet}). We answer this question
(Theorem~\ref{thm:not-closed}) by showing that there are self-full ceers $X,
Y$ such that $X\oplus Y$ is not self-full. Notice that both $X$ and $Y$ must
be infinite since $F \oplus X$ is self-full
whenever $F$ is finite and $X$ is self-full \cite[Corollary~4.3]{joinmeet}.

Motivated by the fact that the uniform join of two self-full ceers need not
be self-full, in the last section of the paper we introduce the notion of a
ceer $X$ being \emph{hereditarily-self-full}, i.e. $X\oplus Y$ is self-full
whenever $Y$ is self-full. Let $\textrm{HSF}$ be the class of \HSF ceers. By
\cite[Corollary~4.3]{joinmeet} we have that $\textrm{HSF}$ contains the finite
ceers. On the other hand we show in Corollary~\ref{cor:hereditarily} that all
dark ceers are hereditarily-self-full. However we show that this inclusion is proper,
i.e. $\dark \subset \textrm{HSF}^{\infty}$ (where the latter class
$\textrm{HSF}^{\infty}$ consists of the infinite \HSF ceers). In Theorem \ref{thm:HSF-Istrongminimalcovers}, we show that every $\I$-degree has infinitely many incomparable \HSF strong minimal covers. In particular, every non-universal degree is bounded by an \HSF degree, so there are light \HSF degrees.
It follows that
\[
\dark\cup \I \subset \textrm{HSF} \subset \SF.
\]

Finally we make two observations, first in Observation~\ref{obs:not-hereditarily} that there is no ``hereditarily
hereditarily self-full'' ceer. That is, for every hereditarily self-full $X$
we can find a self-full ceer $Y$ so that $X\oplus Y$ is not hereditarily
self-full. Secondly in Observation~\ref{HSFclosed} that the hereditarily
self-full ceers are closed under $\oplus$ so if the $Y$ above is hereditarily
self-full (instead of just self-full) then $X\oplus Y$ is hereditarily
self-full as well. Thus the \HSF ceers can, in future analysis, often replace the self-full ceers in cases where a class which is closed under $\oplus$ is needed.

Throughout the paper, if $R$ is an equivalence relation on $\omega$, and
$x\in \omega$, then $[x]_R$ will denote the $R$-equivalence class of $x$. In the proofs of Theorem~\ref{thm:not-closed} and Theorem~\ref{thm:light-hsf}
we will be building ceers, satisfying certain requirements, using the
``collapsing'' technique. That is, when building a ceer $R$ in stages, we start with
$R_{0}=\Id$. The ceer $R_{s+1}$ will be an extension of  $R_{s}$ obtained by
$R$-collapsing at $s+1$ finitely many pairs of numbers $(k,k')$, where we say
that we \emph{$R$-collapse $k,k'$} at $s+1$ if we merge together in $R_{s+1}$
the $R_{s}$-equivalence classes $[k]_{R_{s}}$ and $[k']_{R_{s}}$, so that $\rel{R_{s+1}}$ is the equivalence relation generated by $R_s$ along with the finitely many pairs of numbers which we $R$-collapse at $s+1$.

\section{Self-fullness is not closed under uniform join}
In this section we show that the collection of self-full ceers is not closed under
$\oplus$.

\begin{theorem}\label{thm:not-closed}
There are self-full ceers $X$ and $Y$ so that $X\oplus Y$ is non-self-full.
\end{theorem}

\begin{proof}
We construct a ceer $Z$, which will be $X\oplus Y$ for ceers $X$ and $Y$. We
let $X$ refer to $Z\restriction{\textrm{Evens}}$ (i.e.\  $x\rel{X}y$ if and
only if $2x\rel{Z} 2y$) and $Y$ refer to $Z \restriction{\textrm{Odds}}$
(i.e.\  $x\rel{Y}y$ if and only if $(2x+1) \rel{Z} (2y+1)$). We enumerate the
ceer $Z$ in stages. At stage $s$ we define a ceer $Z_s$ so that $Z=\bigcup_s
Z_s$ is our desired final ceer.
We build $Z_{s+1}$ extending $Z_{s}$ by the collapsing technique
described at the end of Section~\ref{sct:introduction}. Finally, we let $X_s=
Z_s \restriction{\textrm{Evens}}$, and similarly $Y_s
=Z_s\restriction{\textrm{Odd}}$.
\emph{$X$-collapsing $k$ and $k'$} at any stage means that we $Z$-collapse
$2k$ and $2k'$. Similarly, \emph{$Y$-collapsing $k$ and $k'$}  means that we
$Z$-collapse $2k+1$ and $2k'+1$.
 	
We also construct a reduction function $f:Z\rightarrow Z$. In addition to
ensuring that $f$ is a reduction and that its image omits the $Z$-class of
$0$, we have the requirements:
	
\begin{itemize}
\item[$SF_{j,k}^X$:] If $\phi_j$ is a reduction of $X$ to $X$, then $[k]_X$
    intersects the range of $\phi_j$.
\item[$SF_{j,k}^Y$:] If $\phi_j$ is a reduction of $Y$ to $Y$, then $[k]_Y$
    intersects the range of $\phi_j$.
\end{itemize}

We will often say that a number $z$ is in $X$ if $z$ is even, and $z$ is in
$Y$ if $z$ is odd. We will also say that some numbers will be
\emph{$X$-bound} and other numbers may be \emph{$Y$-bound}.	
Namely, when we say that a number $u$ is $X$-bound, this means that if $l$ is
the greatest number such that $f^{(l)}(u)$ has been already defined then as
we choose further $f$-images of the number, i.e. we define $f^{(m)}(u)$ for
$m> l$, we will choose from $X$ (in other words, the image will be an even
number). When we say that $u$ is $Y$-bound, then we will choose from $Y$ (in other words, the image will be an odd
number).  	

We describe a module called $\dmpar{x}$ that we will use to ensure that $\phi_j$ is not a
reduction of $X$ to $X$. We call this module (at some stage $s$) when we have
an element $x$ such that $2x$ is $Y$-bound and we have $\phi_j(x)\downarrow$ and
$\phi_j(x) \cancel{\rel{X_s}}x$.
We may also use this module to ensure that $\phi_j$ is not a reduction of $Y$
to $Y$ if we have an element $y$ such that $2y+1$ is $X$-bound and we have
$\phi_j(y) \downarrow$ and $\phi_j(y) \cancel{\rel{Y_s}}y$.
Though we discuss the diagonalization module for $X$, and the strategy for
$SF_{j,k}^X$, everything is symmetric for $Y$.

\begin{remark}
When dealing with $SF^X$-requirements, we let $\hat{f}(u)= \frac{f(2u)}{2}$,
if $f(2u) \in X$. We define $\hat{f}^{(n)}(u)=v$ if and only if
$f^{(n)}(2u)=2v$.
Notice that we abuse notation in that $\hat{f}^{(n)}$ may properly contain
the $n$th iterate of the function $\hat{f}$. If $\hat{f}^{(n)}(u)$ and
$\hat{f}^{(n)}(v)$ are both defined, then it is easy to see, using that $f: X
\oplus Y \longrightarrow X \oplus Y$ is a reduction, that $u \rel{X} v$ if
and only if $\hat{f}^{(n)}(u) \rel{X} \hat{f}^{(n)}(v)$. Symmetrically, when
dealing with $SF^Y$-requirements, we let $\hat{f}(u)= \frac{f(2u+1)-1}{2}$.
In this case $\hat{f}^{(n)}(u)=v$ if and only if $f^{(n)}(2u+1)=2v+1$.
\end{remark}

\subsubsection*{Informal description of $\dmpar{x}$.}

\subsubsection*{Goal 1}
The first goal of the module is to find some
 element $z$ (which might not be
$x$)
so $2z$ is $Y$-bound
and a stage $t>s$ so that for every $n$ such that $\hat{f}^{(n)}(z)$ is
defined, $\phi_j(z)\cancel{\rel{X_t}} \hat{f}^{(n)}(z)$. More precisely, let
$S$ be the set of $n$ for which, by stage $s$, $f^{(n)}(2x)$ is determined
and in $X$. Notice that no other $f^{(m)}(2x)$ is in $X$ (since
$2x$ is $Y$-bound so all later choices of $f^{(m)}(2x)$ will be odd). If
$\phi_j(x) \cancel{\rel{X_{s}}} \hat{f}^{(n)}(x)$ for any $n \in S$, then we
let $z=x$, and we move to Goal 2.

Otherwise, wait for a stage $t>s$ such that at this stage
$\phi_j(\hat{f}^{(n)}(x))$ converges for each $n \in S$. As we will argue in
Lemma~\ref{lem:main-satisfaction}, either we see at this point that
$\phi_{j}$ is not a reduction on the elements $\{\hat{f}^{(n)}(x): n\in S\}$,
or we see that there is some $m \in S$, $m>0$, such that, taking
$z=\hat{f}^{(m)}(x)$, we have that $\phi_j(z)\cancel{\rel{X_t}}
\hat{f}^{(n)}(z)$, for every $n$ so that $\hat{f}^{(n)}(z)$ is defined. After
this we move to Goal~2.
	
\subsubsection*{Goal 2}
We are given a $z$
so that $2z$ is $Y$-bound and
 so that $\phi_j(z)$ converges by stage $t$ and is not
$X_t$-equivalent to any $\hat{f}^{(n)}(z)$ which is already defined. We
consider a new element $w$ and wait for a stage $s'>t$ at which $\phi_j(w)$
converges. In the meantime, we ensure that $f^{(n)}(2w)$ is in $X$ if and
only if $f^{(n)}(2z)$ is in $X$. The fact that $w$ is new will ensure that
the class of $2w$ does not intersect the range of $f$. Thus we may later
$X$-collapse $w$ with $z$ (by $Z$-collapsing $2w$ with $2z$) if we so wish.
Once $\phi_j(w)$ converges at $s'$, if $\phi_j(w)\cancel{\rel{X_{s'}}}
\phi_j(z)$, then we $X$-collapse $w$ with $z$ (i.e., we $Z$-collapse
$2w$ with $2z$). Since $\phi_j(z)$ is not $X$-equivalent to any element
$\hat{f}^{(n)}(z)$ and the class of $2w$ does not
intersect the image of $f$, this does not cause $\phi_j(z)$ to $X$-collapse
with $\phi_j(w)$. Indeed, the $X$-collapse of $w$ with $z$ entails only (as
$f$ is a reduction from $Z$ to $Z$) the $Z$-collapse of each $f^{(n)}(2z)$
with $f^{(n)}(2w)$, i.e., the $X$-collapse of $\hat{f}^{(n)}(z)$ with
$\hat{f}^{(n)}(w)$. Therefore $\phi_j(z)$ is not $X$-collapsed to any element
as a consequence of the execution of Goal~2.

We have thus diagonalized to ensure that $\phi_j$ is not a reduction of $X$
to $X$. 	

\medskip

We now describe the strategy to satisfy $SF_{j,k}^X$. The strategy to satisfy
$SF_{j,k}^Y$ is symmetric.
	
\subsubsection*{Informal description of the $SF_{j,k}^X$-strategy}
First, suppose that $2k$ is either not bound by a higher-priority requirement
or is $Y$-bound by a higher-priority requirement. If it is not bound by a
higher-priority requirement, then we make $2k$ $Y$-bound. When the strategy
is called at stage $s$, we may have already defined $f(2k), \ldots,
f^{(l)}(2k)$ for some $l$. This means that we will choose $f^{(m)}(2k)$ to be
in $Y$ (i.e.,  odd) for $m>l$. Next we wait for $\phi_j(k)$ to
converge, at say stage $s$. If $\phi_j(k) \rel{X_s} k$, then we have
satisfied the requirement. Otherwise, we use $\dmpar{k}$ to ensure that
$\phi_j$ is not a reduction of $X$ to $X$.

Next, we consider the case where $2k$ is $X$-bound by a higher-priority
requirement, thus we cannot use the previous strategy. We begin by choosing a
new element $k'$, and we ensure that if we have defined $f(2k), \ldots,
f^{(l)}(2k)$ already, then we define $f^{(m)}(2k')$ for $m\leq l$ to be in
$X$ if and only if $f^{(m)}(2k)$ is in $X$, and we make $2k'$ $X$-bound.
Next, we wait for $\phi_j(k')$ to converge, at, say, stage $s$. If
$\phi_j(k')$ converges and is already $X$-equivalent to $k'$, then we
$X$-collapse $k'$ with $k$ and we are done. Otherwise, $\phi_j(k')$ converges
and is not $X_s$-equivalent to $k'$. In this case, we determine that $2k'$ is
$Y$-bound and call $\dmpar{k'}$ to ensure that $\phi_j$ is not a reduction of
$X$ to $X$.

\subsection*{Construction:}
As anticipated at the beginning of the proof, we enumerate our desired ceer
$Z$ in stages, so that  at stage $s$, $Z_s$ will be a ceer, and we let $X_s$ be
$Z_s\restriction{\textrm{Evens}}$ and let $Y_s$ be
$Z_s\restriction{\textrm{Odds}}$. 		

We say that a requirement $SF^X_{j,k}$ or $SF^Y_{j,k}$ \emph{requires
attention at stage $s$} if it has not acted since it was last initialized, or
if it sees some computations $\phi_j(u)$ converge for some numbers $u$, and
it had been waiting for these computations to converge. Requirements may
determine that numbers are \emph{$X$-bound} or \emph{$Y$-bound}. When we
write that a strategy makes a number $k$ $X$-bound or $Y$-bound, it
simultaneously makes $f^{(n)}(k)$ $X$-bound (or $Y$-bound) for every $n$ so
that $f^{(n)}(k)$ is already defined. If no active strategy is making a
number $X$-bound or $Y$-bound, then we say that the number is \emph{free}. A
number is \emph{new} at stage $s+1$,
if it is bigger than $s+1$ and none of its $X_s$-, $Y_s$-, or $Z_s$-equivalence classes contains any number so far used in the construction.

We \emph{initialize} a requirement $SF^X_{j,k}$ or $SF^Y_{j,k}$ by reverting it back to the beginning of its strategy, i.e., to the initial
distinction between Case 1 and Case 2 described below.
 	
\subsubsection*{Stage $0$} All requirements are initialized. Let $Z_0=\Id$.

\subsubsection*{Stage $s+1$}
Let $SF^X_{j,k}$ (or $SF^Y_{j,k}$, for which we act symmetrically, by just
replacing $X$ with $Y$ and the even numbers with the odd numbers) be the
highest-priority requirement which requires attention at stage $s+1$. Since
almost all requirements are initialized there is such a
requirement. Then we re-initialize all requirements of lower-priority, and
act as follows:

\medskip

-- The requirement requires attention because it is initialized. We
distinguish the following two cases:

\paragraph{\emph{Case $1$}}
The number $2k$ is free or $Y$-bound. If $2k$ is currently free, then we make
$2k$ \emph{$Y$-bound}. Wait for a stage $t>s+1$ where $\phi_j(k)$ converges
and $\phi_j(k)\cancel{\rel{X_t}}k$. We say that the strategy has entered the \emph{Case~1-waiting outcome}.

\paragraph{\emph{Case $2$}}
The number $2k$ is $X$-bound.

\paragraph{\quad \emph{Step $0$} of \emph{Case $2$}}
Let $k'$ be a new element. Let $n$ be greatest so that $f^{(n)}(2k)$ is
already defined. Define $f^{(m)}(2k')$ to be new elements for each $m\leq n$
ensuring that $f^{(m)}(2k')$ is in $X$ if and only if $f^{(m)}(2k)$ is in
$X$. Make $2k'$ \emph{$X$-bound}. Wait for $\phi_j(k')$ to converge. We say that the strategy has entered
the \emph{Case~2 Step~0-waiting outcome}.

\medskip
-- The requirement requires attention because it was in the Case~1-waiting
outcome or in the Case~2 Step~0-waiting outcome, and now the awaited
computations have converged. If it was in the Case~1-waiting outcome then we
will return to Step~1 of $\dmpar{x}$ (as described below) with $x:=k$. If it
was in the Case~2 Step~0-waiting outcome then it will return to Case $2$ Step
$1$.

\paragraph{\quad \emph{Step $1$} of \emph{Case $2$}}
If $\phi_j(k') \rel{X_s} k'$, then $X_{s+1}$-collapse $k'$ with $k$. Declare
the requirement \emph{satisfied}. Otherwise, make $2k'$ \emph{$Y$-bound} and
go to Step~1 of $\dmpar{x}$ (as described below) with $x:=k'$. 		

\medskip
-- If the requirement requires attention because it was in the $\dmpar{x}$
Step~1-waiting outcome, and now the awaited computation has converged, then
it will return to Step~2 of $\dmpar{x}$:

\medskip
-- If the requirement requires attention because it was in the $\dmpar{x}$
Step~$3(z)$-waiting outcome, and now the awaited computation has converged
then it will return to Step~$4(z,w)$ of $\dmpar{x}$.

\subsubsection*{$\emdmpar{x}$}
(Called at $s+1$  for a $Y$-bound $2x$, with $\phi_{j}(x)\downarrow$ and
$\phi_j(x)\cancel{\rel{X_s}} x$.)

\begin{itemize}
\item[] \quad \emph{Step~1}. We define $S$ be the set of $n$ so that
    $f^{(n)}(2x)$ is defined by stage $s$ and is in $X$. Notice that since
    $2x$ is $Y$-bound then we will not define any $f^{(m)}(2x)$ to be in
    $X$ at a later stage. If $\phi_j(x)\cancel{\rel{X_s}} \hat{f}^{(n)}(x)$
    for each $n\in S$, then we go to Step~$3(x)$. Otherwise, wait for
    $\phi_j(\hat{f}^{(n)}(x))$ to converge for every $n\in S$. We say the module has entered the
    \emph{$\emdiagmod$ Step~1-waiting outcome}.

\item[] \quad \emph{Step~2}. If the wait for the computations
    $\phi_j(\hat{f}^{(n)}(x))$ in Step~1  to converge is over then pick
    some $z=\hat{f}^{(n)}(x)$ for $n\in S$ so that $\phi_j(z)
    \cancel{\rel{X_s}} \hat{f}^{(m)}(z)$ for every $m$ such that
    $f^{(m)}(2z)$ is determined before stage $s$ and is in $X$. Then go to
    Step~$3(z)$. We will argue below (in the proof of
    Lemma~\ref{lem:main-satisfaction}) that either the requirement is
    satisfied (i.e., we see some $n$ so that $\phi_j(n)\rel{X} k$, or
    $\phi_j$ is not a reduction of $X$ to $X$) or that such a $z$ must
    exist.
Note that since $2z=f^{(n)}(2x)$, $2z$ is also $Y$-bound.
\item[] \quad \emph{Step~$3(z)$}. Take a new element $w$. For every $m$ for
    which $f^{(m)}(2z)$ is defined, define $f^{(m)}(2w)$ to be a new number
    so that $f^{(m)}(2w)$ is in $X$ if and only if $f^{(m)}(2z)$ is in $X$.
Make $2w$ be $Y$-bound (since $2z$ is $Y$-bound this preserves our ability to $X$-collapse $w$ with $z$).
Wait for $\phi_j(w)$ to converge. We say the module has entered the \emph{$\emdiagmod$ Step~$3(z)$-waiting outcome}.

  \item[] \quad \emph{Step~$4(z,w)$}. If the wait for the computation
      $\phi_j(w)$ in Step~$3(z)$ to converge  is over then we have
      $\phi_j(z)$ and $\phi_j(w)$ both converged by stage $s$. If
      $\phi_j(z) \rel{X_s} \phi_j(w)$, then we do nothing and declare the
      requirement \emph{satisfied}. Otherwise, we collapse $z \rel{X_{s+1}}
      w$ and declare the requirement \emph{satisfied}. 	
\end{itemize}

We choose the first number $v$ on which we have not defined $f(v)$. If $v$ is $X$-bound, we take a new number $x$ in $X$ and define $f(v)=x$.
Otherwise, we take a new number $y$ in $Y$ and define $f(v)=y$. If $v$ was
$X$-bound (or $Y$-bound), then we also make $f(v)$ be $X$-bound (or
$Y$-bound).

Let $Z^0_{s+1}$ be the equivalence relation generated by $Z_s$ along with any pairs that we have decided to $Z$-collapse during this stage.

Lastly, we do further collapses in order to ensure $f$ is a reduction. That is, we let $Z_{s+1}$ be the result of closing $Z^0_{s+1}$ under the implication: If $z \rel{Z_{s+1}} w$ and $f(z)$ and $f(w)$ are defined, then $f(z)\rel{Z_{s+1}}f(w)$. Note that since $f$ is only defined on finitely many values, this causes finitely much further collapse. We then stop the stage and go to stage $s+2$.

\subsection*{Verification:}

We will adopt the following notation throughout the verification: Given a
parameter $x$ chosen by an $SF$-strategy $R$, we will write $\widetilde{x}
=2x$ if $R=SF^X_{j,k}$ for some $j,k$ and $\widetilde{x}=2x+1$ if
$R=SF^Y_{j,k}$ for some $j,k$.

\begin{lemma}\label{finiteReinitialization}
Each requirement is re-initialized only finitely often.
\end{lemma}

\begin{proof}
Each requirement can act to re-initialize lower-priority requirements,
without being re-initialized itself, at most finitely often.
\end{proof}

In the following definition, an $SF$-strategy or a  $\diagmod$ module is
called \emph{active} for a parameter $x$ if the module has already chosen
this parameter and is waiting for a computation involving this parameter to
converge.

\begin{emdef}
A \emph{time} during the construction is either the beginning of a stage or immediately follows any $Z$-collapse or the assignment of any new parameter or unassignment (via initialization) of any parameters.

If $s$ is a stage, we abuse notation letting $s$ also represent the time that is the beginning of the stage $s+1$. Note that this agrees with our notation of letting $Z_s$ be the ceer we have constructed by the end of stage $s$, which is the same as the beginning of stage $s+1$.

At each time $t$ in the construction, let $D_t$ be the set of elements $\widetilde{w}$ which
refer to parameters $w$ for an active $\diagmod$ which is waiting in
Step~$3(z)$. 		

For each time $t$, let $E_t$ be the set of elements $\widetilde{k'}$ which
refer to parameters $k'$ for an active $SF$-strategy which is waiting in Case
$2$ Step~0. 	

For each time $t$, let $I_t$ be the set of elements which are in the image
of $f$ at time $t$.

For a time $t$, we let $Z_t$ be the ceer $Z$ as it appears at time $t$. That is, if $t$ is during stage $s+1$, $Z_t$ is the equivalence relation generated by $Z_s$ and any $Z$-collapse so far done during stage $s+1$. Note that if $s$ is a stage, this definition of $Z_s$ agrees with the previous definition.

We say that a number $x$ is \emph{new at time $t$} (a time during stage $s+1$) if $x>s+1$ and none of its $X_s$-, $Y_s$-, or $Z_s$-equivalence classes contain any number so far used in the construction.
\end{emdef}

\begin{observation}
At each time $t$, the sets $D_t$, $E_t$ and $I_t$ are disjoint.
\end{observation}
\begin{proof}
Each of the three types of parameters: $w$ for an active $\diagmod$ which is
in Step 3($z$), $k'$ for active $SF$-strategies in Case 2 Step 0, and
elements $f(n)$, are chosen to be new. Thus it is impossible that the same
number is in two of these sets.
\end{proof}

\begin{lemma}\label{unwantedCollapse}
For every time $t$, if $x,y\in D_t\cup E_t\cup I_t$ and $x \rel{Z_t} y$, then $x=f(n)$ and $y=f(m)$ and $n \rel{Z_{t}} m$ for some $n,m$ where $f(n)$ and $f(m)$ are defined at time $t$.

For any times $t\leq t'$ (i.e., $t$ occurs before $t'$ in the construction):
If $x$ is a number which is not new at time $t$ and $x$ is not
$Z_{t}$-equivalent to any member of $D_{t}\cup E_{t}\cup I_{t}$, then $x$ is not
$Z_{t'}$-equivalent to any member of $D_{t'}\cup E_{t'}\cup I_{t'}$.
\end{lemma}

\begin{proof}
We prove both claims by simultaneous induction on times $s$. We assume the claims hold for all times $t,t'\leq s$. Let $s'$ be the next time (i.e., exactly one action happens between $s$ and $s'$). We verify that they still hold where $t,t'$ are assumed to be times $\leq s'$. The claims clearly hold for $s$ being the beginning of the construction (i.e. $Z_s=\Id$ and no parameters have been chosen). Exactly one action takes place between times $s$ and $s'$. There are three possibilities for this action:
\begin{itemize}
	\item We may remove parameters via re-initialization or choose a new parameter which enters $D_{s'}\cup E_{s'}\cup I_{s'}$.
	\item We can cause $Z$-collapse to satisfy a requirement. This happens either between $\widetilde{k}$ and $\widetilde{k'}$ in Case 2 Step 1 of a $SF_{j,k}^X$- or $SF_{j,k}^Y$-strategy or between $\widetilde{w}$ and $\widetilde{z}$ in Step~$4(z,w)$ of a $\diagmod$.
	\item We may $Z$-collapse $f(z)$ with $f(w)$ at the end of the stage because we have already $Z$-collapsed $z$ with $w$.
\end{itemize}

We first consider removing or choosing new parameters. Removing parameters clearly maintain both claims since there are strictly fewer active numbers for the claims to hold for. Since all new parameters are chosen to be new numbers, their classes are singletons. Thus both claims are maintained by assigning new parameters.

We next consider $Z$-collapse between $\widetilde{k}$ and $\widetilde{k'}$ in Case 2 Step 1 of a $SF_{j,k}^X$- or $SF_{j,k}^Y$-strategy. By the first claim of the inductive hypothesis, before this collapse $\widetilde{k'}$ was the only member of its $Z_s$-class which was in $D_s\cup E_s\cup I_s$. Since we make this collapse and simultaneously make the requirement satisfied (i.e., inactive), $\widetilde{k'}$ is not in $D_{s'}\cup E_{s'}\cup I_{s'}$ after this action. Thus to the class of $\widetilde{k}$ we have added no member of $D_{s'}\cup E_{s'}\cup I_{s'}$. We also note that the collection of numbers which are now equivalent to active numbers is a subset of the numbers that were equivalent to active numbers before the collapse. Thus this action did not make either of the claims false.

Next we consider $Z$-collapse between $\widetilde{w}$ and $\widetilde{z}$ in Step~$4(z,w)$ of a $\diagmod$. Once again, $\widetilde{w}$ is no longer active after this collapse since the strategy is now satisfied and no longer active. By the first claim of the inductive hypothesis, before this collapse $\widetilde{w}$ was the only member of its $Z_s$-class which was in $D_s\cup E_s\cup I_s$. As in the previous case, this maintains both claims.

Finally, we consider what happens when we collapse $f(z)$ with $f(w)$ once we have already collapsed $z$ with $w$. Suppose this caused $Z_{s'}$-collapse between $x$ and $y$ for $x,y\in D_{s'}\cup E_{s'}\cup I_{s'}$. Note that this action does not change the values of these sets, so $x,y\in D_s\cup E_s\cup I_s$. By possibly re-naming $x$ and $y$, we may assume $x \rel{Z_s} f(z)$ and $y\rel{Z_s} f(w)$. By inductive hypothesis, this is only possible if $x=f(n)$ and $y=f(m)$ where $n \rel{Z_s} z$ and $m\rel{Z_s} w$. So we see that $n\rel{Z_{s'}}z\rel{Z_{s'}}w\rel{Z_{s'}}m$ and the first claim is preserved. Similarly the second claim is preserved since we did not change the $Z$-closure of the set $D_{s'}\cup E_{s'}\cup I_{s'}=D_s\cup E_s\cup I_s$.
\end{proof}

\begin{lemma}
If $x$ is even and $y$ is odd, then we never $Z$-collapse $x$ with $y$. That is,
$Z=X\oplus Y$.
\end{lemma}

\begin{proof}
We cause collapse in two cases: The first is $Z$-collapsing $\widetilde{k}'$
with $\widetilde{k}$ in Case 2 Step~1. In this case, since we defined $f$ on
$\widetilde{k'}$ so that $f^{(m)}(\widetilde{k'})$ is in $X$ if and only if
$f^{(m)}(\widetilde{k})$ is in $X$, and both $\widetilde{k}$ and
$\widetilde{k'}$ were $X$-bound, or both $Y$-bound,
when we $Z$-collapse $\widetilde{k}$ with
$\widetilde{k'}$ (even after closing under making $f$ a reduction) we do not cause any $Z$-collapses between even and odd numbers.

The second case is $Z$-collapsing $\widetilde{z}$ with $\widetilde{w}$ in
Step~$4(z,w)$. In this case, we have made $\widetilde{z}$ $Z$-collapse with
$\widetilde{w}$ where we have ensured that $f^{(m)}(\widetilde{w})$ is in $X$
if and only if $f^{(m)}(\widetilde{z})$ is in $X$ and then we made both
$\widetilde{z}$ and $\widetilde{w}$ be $X$-bound or both be $Y$-bound. Thus
we see that we do not cause any $Z$-collapse between even and odd numbers.
\end{proof}

\begin{lemma}
$f$ is a reduction of $X\oplus Y$ to $X\oplus Y$ and $[0]_{X\oplus Y}$ is not
in the range of $f$.
\end{lemma}

\begin{proof}
Since whenever we have $z Z_s w$, we also cause $f(z) Z_s f(w)$, it remains
to see that if $f(z)Z_s f(w)$, then we also have $z Z_s w$. This follows from the first claim of
Lemma \ref{unwantedCollapse}.

By the second claim in Lemma \ref{unwantedCollapse} and the fact that 0 is not new even at stage $0$, $0$ never becomes $Z$-equivalent to a number in the image of $f$.
\end{proof}

\begin{lemma}\label{unwantedCollapse2}
	Suppose $x$ and $y$ are mentioned by a strategy $R$ at stage $s$ and
	$x \cancel{\rel{Z_s}} y$. Suppose that $t>s$ is a stage so
	that the strategy $R$ has not been re-initialized or acted at any stage
	between $s$ and $t$. Then $x \cancel{\rel{Z_t}} y$.
\end{lemma}

\begin{proof}
	It suffices to show that no lower-priority requirement can cause a collapse which will make $x$ $Z$-collapse with $y$. We begin by showing that any number which is not new at stage $s$ (such as $x$ or $y$) does not become $Z_{t'}$-equivalent to parameters $\widetilde{k'}$ for a strategy of lower-priority
	than $R$ in Case 2 Step~0 or parameters $\widetilde{w}$ for a
	$\diagmod$ for a strategy of lower-priority
	than $R$ which is waiting in Step~$3(z)$ for any time $s<t'\leq t$. Similarly, if $a$ is not new at stage $s$ and is not $Z_s$-equivalent to a member of $I_s$, then $a$ is not $Z_{t'}$-equivalent to a member of $I_{t'}$ for any time $s<t'\leq t$.
	
	Let $t'$ be a time so that $s<t'\leq t$, and let $a$ be not new at stage $s$. At stage $s$, all lower-priority requirements are re-initialized. Thus either
	$a$ is not $Z_s$-equivalent to a member of $D_s\cup E_s\cup I_s$, so the second claim of
	Lemma \ref{unwantedCollapse} ensures that it is  not $Z_{t'}$-equivalent to a
	member of $D_{t'}\cup E_{t'}\cup I_{t'}$, or it is
	$Z_s$-equivalent to a member of $D_s\cup E_s\cup I_s$.
	
	 If $a$ is $Z_s$-equivalent to
	a member
	$u\in D_s\cup E_s$, then $u$ is a parameter for a
	requirement of priority at least as high as $R$, which does not act or get
	reinitialized before stage $t$, so $u\in D_{t'}\cup E_{t'}$. Similarly,
	if $u\in I_s$, then $u\in I_{t'}$, since $I_s\subseteq
	I_{t'}$. The first claim of Lemma~\ref{unwantedCollapse} shows that
	$u$ is not $Z_{t'}$-equivalent to parameters $\widetilde{k'}$ for a lower-priority
	requirement in Case 2 Step~0, or parameters $\widetilde{w}$ for a lower-priority
	$\diagmod$ which is waiting in Step~$3(z)$. And if $u\in D_s\cup E_s$, the first claim of Lemma~\ref{unwantedCollapse} also shows that $u$ is not $Z_{t'}$-equivalent to a member of $I_{t'}$.
	We conclude that any number $a$ which is not new at stage $s$ is never equivalent to a parameter $\widetilde{k'}$ for a lower-priority requirement in Case 2 Step~0 or
	parameters $\widetilde{w}$ for a lower-priority $\diagmod$ which is waiting in
	Step~$3(z)$ at any $t'$ between $s$ and $t$. Further, if $a$ is not $Z_s$-equivalent to a member of $I_s$, then $a$ is not $Z_{t'}$-equivalent to a member of $I_{t'}$.	
	
	We define a unique finite sequence of classes $[a_0]_{Z_s},\ldots, [a_n]_{Z_s}$ as follows. We let $a_0=x$. Having defined a class $[a_i]_{Z_s}$, we let the class $[a_{i+1}]_{Z_s}$ be the (unique by the first claim in Lemma \ref{unwantedCollapse}) class of a number $u$ so that $f(u)\in [a_i]_{Z_s}$. If no such class exists, we simply end the finite sequence. We similarly define a sequence of classes $[b_i]_{Z_s}$ beginning with $b_0=y$.
	
	\begin{claim}
		Let $t'$ be a time so $s<t'\leq t$. Suppose that $c$ is a number so that $f^{(k)}(c)$ is defined at time $t'$ and is $Z_{t'}$-equivalent to $x$. Then $c\rel{Z_t} a_i$ for some $i$. Similarly for $y$ and the $b_j$.
	\end{claim}
	\begin{proof}
		We distinguish between two cases. Case 1: $k> n$. Then by the first claim of Lemma \ref{unwantedCollapse} applied $n$ times, $f^{k-n}(c) \rel{Z_{t'}} a_n$. But $a_n$ is not new at stage $s$ and is not $Z_s$-equivalent to a member of $I_s$, so it is not $Z_{t'}$-equivalent to a member of $I_{t'}$. This is a contradiction. Case 2: $k\leq n$. Then applying the first claim of Lemma \ref{unwantedCollapse} $k$ times, we see that $c\rel{Z_{t'}}a_k$.
	\end{proof}
	
	Since no member of $[a_i]_{Z_s}$ is new at stage $s$ for any $i\leq n$, we know that they cannot be $Z_{t'}$-equivalent to a parameter $\widetilde{k'}$ for a lower-priority requirement in Case 2 Step~0 or
	parameters $\widetilde{w}$ for a lower-priority $\diagmod$ which is waiting in
	Step~$3(z)$ at any $t'$ between $s$ and $t$. Similarly for the $[b_j]_{Z_s}$. This along with the previous claim shows that whenever a lower-priority requirement acts, it cannot collapse the classes of $a_i$'s with $b_j$'s. In particular, even after we ensure that $f$ is a reduction (i.e., we collapse $f(z)$ with $f(w)$ if we have already collapsed $z$ with $w$), we will not collapse $x$ with $y$.
\end{proof}
	
\begin{lemma}\label{lem:main-satisfaction}
Each requirement is satisfied, so both $X$ and $Y$ are self-full.
\end{lemma}

\begin{proof}
Let us consider a requirement $SF^X_{j,k}$. By Lemma
\ref{finiteReinitialization}, we can fix a stage $s$ to be the last time this
strategy will ever be re-initialized. After stage $s$, when it next chooses
parameters $k'$, $z$, and $w$, these choices are permanent. If we are in Case
1, then we either have $\phi_j(k)$ diverges, or it converges and $\phi_j(k)
\rel{X} k$, or we go to $\dmpar{k}$. In the first two cases, we have
satisfied the requirement because either $\phi_j$ is not total, or $[k]_X$
intersects the image of $\phi_j$. We consider the last case below. Similarly,
if we are in Case 2, then we either have $\phi_j(k')$ diverge or $\phi_j(k')
\rel{X_s} k'$ at the next stage $s$ when the requirement acts, in which case
we cause $k' \rel{X} k$, or we enter $\dmpar{k'}$. In the first two cases the
requirement is satisfied because again either $\phi_j$ is not total or
$[k]_X$ intersects the image of $\phi_j$.

So, we can suppose that a $\diagmod$ is begun. If we wait in Step~1 forever,
this guarantees that $\phi_j$ is not total and the requirement is satisfied.
Otherwise at some stage $s+1$ we begin Step~2. We claim that in Step~2,
either the requirement is satisfied or we must succeed in picking some $z$.
Consider the finite set of $X$-classes $F=\{[\hat{f}^{(n)}(x)]_X\mid n\in
S\}$. If $\phi_j$ is a reduction of $X$ to $X$, then it must be injective on
the classes in $F$. So, it either sends some class in $F$ to a class outside
of $F$ or it is surjective on $F$. In the latter case, $[x]_X=
[\hat{f}^{(0)}(x)]_X$ intersects the range of $\phi_j$ and the requirement is
satisfied. In the former case, there is some $n\in S$ so that
$\phi_j(\hat{f}^{(n)}(x))\cancel{X} \hat{f}^{(m)}(x)$ for every $m$. Thus we
will find the needed $z=\hat{f}^{(n)}(x)$. Thus we begin Step~3. If we are
stuck in Step~3, then $\phi_j$ is not total. If we get to step~$4(z,w)$ at
say stage $s'+1$, we have two cases to consider: If $\phi_j(z) \rel{X_{s'}}
\phi_j(w)$, then we do not
$X$-collapse $z$ with $w$. This guarantees that $z\cancel{\rel{X_t}}w$ for
all $t>s'$ by Lemma \ref{unwantedCollapse2}, therefore $\phi_j$ is not a
reduction. On the other hand, if we have $\phi_j(z)\cancel{\rel{X_{s'}}}
\phi_j(w)$, then we
$X$-collapse $z
\rel{X_{s'+1}} w$, which does not cause $\phi_j(z)\rel{X_{s'+1}}\phi_j(w)$. By
Lemma \ref{unwantedCollapse2}, we have $\phi_j(z)\cancel{\rel{X}}\phi_j(w)$,
and $\phi_j$ is not a reduction.
\end{proof}
The proof is now complete.
\end{proof}

\section{The \HSF ceers}
We introduce the \emph{hereditarily self-full ceers} and show that they properly
contain the dark ceers.

\begin{emdef}
A ceer $X$ is \emph{hereditarily self-full} if whenever $Y$ is self-full,
then $X\oplus Y$ is self-full.
\end{emdef}

Notice that by~\cite[Corollary~4.3]{joinmeet} every finite ceer is \HSF.

The next definition highlights a property, which, when accompanied by
self-fullness, is sufficient to guarantee hereditary self-fullness as proved
in Theorem~\ref{SFCCRsAreHSF} below.

\begin{emdef}
We say that a ceer $X$ is \emph{co-ceer-resistant} if whenever $C$ is a
$\Pi^0_1$-equivalence relation with infinitely many classes, then there are
$x,y$ so that $xXy$ and $x\cancel{\rel{C}}y$.
\end{emdef}

\begin{theorem}\label{SFCCRsAreHSF}
Every self-full \CCR ceer is \HSF.
\end{theorem}

\begin{proof}
Given any reduction $f: X\oplus Y\rightarrow X\oplus Y$, define $n
\rel{\sim^f} m$ if for every $k\in \omega$, $f^{(k)}(2n)$ is even if and only
if $f^{(k)}(2m)$ is even. Note that $\sim^f$ is a $\Pi^0_1$-equivalence
relation and $X$ is a refinement of $\sim^f$ (i.e.\ if $n \rel{X} m$ then $n
\rel{\sim^f} m$). We represent a $\sim^f$-class by a sequence in
$\{X,Y\}^{\omega}$. Namely, we define $\tau^f_n \in \{X,Y\}^{\omega}$ by
$\tau^f_n(k)=X$ if $f^{(k)}(2n)$ is even (or $f^{(k)}(2n) \in X$, as we will
sometimes write), and $\tau^f_n(k)=Y$ if $f^{(k)}(2n)$ is odd (or
$f^{(k)}(2n) \in Y$, as we will sometimes write). It is immediate to see that
$n \rel{\sim^f} m$ if and only if $\tau^f_n=\tau^f_m$, thus in fact we can
identify the $\sim^f$-equivalence class of $n$ with $\tau^f_n$.

A sequence $\tau \in \{X,Y\}^{\omega}$ is \emph{eventually periodic} if there
exist finite sequences $\rho, \sigma \in \{X,Y\}^{<\omega}$ so that
$\tau=\rho \concat\sigma^\infty$, where the symbol $\concat$ denotes
concatenation, and $\sigma^\infty=\sigma \concat \sigma \concat \cdots
\concat \sigma \cdots$ is the infinite string obtained by concatenating
infinitely many times $\sigma$ with itself. Such a string $\sigma$ is a
\emph{period} of $\tau$.

We say that a $\sim^f$-class is \emph{periodic of period $k$} if the infinite
string $\tau$ which represents $\sim^f$ is periodic with a period of length
$k$.

Let now $X$ be a self-full \CCR ceer and suppose $Y$ is self-full. We must
show that $X\oplus Y$ is self-full. Suppose that $f:X\oplus Y\rightarrow
X\oplus Y$ is a reduction. We must show that the range of $f$ intersects all
the equivalence classes of $X\oplus Y$. Since $X \subseteq \sim^f$ and $X$ is
\CCR, it follows that $\sim^f$ has only finitely many classes.

\begin{lemma}
Each of the finitely many $\sim^{f}$-classes is represented by an eventually
periodic sequence $\tau\in \{X,Y\}^{\omega}$.
\end{lemma}

\begin{proof}
Let $\tau$ be a sequence which represents a $\sim^f$-class, say $\tau=
\tau^f_n$. If $\tau=\rho \concat Y^\infty$ for some finite $\rho$ (where, to
be more precise, $Y^\infty= \langle Y\rangle ^\infty$, where $\langle Y
\rangle$ denotes the string of length one consisting of the sole bit $Y$)
then $\tau$ is eventually periodic. Otherwise, let $\{k_i\}_{i \in \omega}$
be the sequence in ascending order so that $\tau(k_i)=X$ (notice that
$k_0=0$). For every $i$, let ${\tau}_{\restriction {\ge k_i}}$ be the
\emph{tail} of $\tau$ beginning from the bit $k_i$ of $\tau$, i.e.
${\tau}_{\restriction {\ge k_i}}(k) = \tau(k_i+k)$. Clearly for every $i$
there exists $n_i$ so that ${\tau}_{\restriction {\ge k_i}}=\tau^f_{n_i}$:
just take $n_i=\frac{f^{(k_i)}(2n)}{2}$. If it were ${\tau}_{\restriction
{\ge k_i}}\ne {\tau}_{\restriction {\ge k_j}}$ for every $i$, and $j<i$, then
there would be infinitely many $\sim^f$-equivalence classes, contrary to our
previous conclusion. Therefore there exist a least $j$, and for this $j$ a
least $i>j$, so that ${\tau}_{\restriction {\ge k_i}}= {\tau}_{\restriction
{\ge k_j}}$. But then we see that there is a period, namely the finite string
$\sigma$ such that ${\tau}_{\restriction \ge k_j}= \sigma
\concat{\tau}_{\restriction \ge k_i}$.
\end{proof}

Let $N_1$ be a common multiple of all periods of $\sim^f$-classes. We replace
$f$ by $f_1=f^{(N_1)}$, which is still a reduction of $X\oplus Y$ to $X\oplus
Y$,  and now all $\sim^{f_1}$-classes have period $1$. It follows that all
the $\tau$'s which represent $\sim^{f_1}$-classes  are of the form $\rho
X^{\infty}$ or $\rho Y^{\infty}$, for some $\rho$. Let $N_2$ be a number
greater than the length of $\rho$ for each $\sim^{f_1}$-class. Once again, we
replace $f_1$ by $f_2=f_1^{(N_2)}$, which ensures that each
$\sim^{f_2}$-class is either represented by $X^\infty$ or $XY^\infty$. Thus
we have at most two $\sim^{f_2}$-classes. Let $C_X^{f_2}=\{2n:
\tau^{f_2}_n=X^\infty\}$ and $C^{f_2}_Y=\{2n: \tau^{f_2}_n=X\concat
Y^\infty\}$.

\begin{lemma}\label{Cx}
The range of $f_2\restriction C_X^{f_2}$ intersects all the $X\oplus
Y$-classes of $C_X^{f_2}$, and there are no odd numbers in
$f_2^{-1}(C_X^{f_2})$.
\end{lemma}

\begin{proof}
First of all, notice that $C_X^{f_2}, C_Y^{f_2}$ partition $2\omega$, they
are both $X\oplus Y$-closed, and they are decidable, being two $\Pi^0_1$-sets
that partition a decidable set. Therefore we can define a reduction $g: X
\rightarrow X$ by:
\[
g(n)=
\begin{cases}
\dfrac{f_2(2n)}{2}, &\textrm{if $2n \in C_X^{f_2}$},\\
n, &\textrm{if $2n \in C_Y^{f_{2}}$}.
\end{cases}
\]
As $X$ is self-full it follows that the range of $g$ must intersect all the
classes of $X$. If it were true that $g(m)=n$ for some $m,n$ such that $2m
\in C_Y^{f_2}$ and $2n \in C_X^{f_2}$, then we would have $m=n$ by definition
of $g$, which contradicts the fact that $C_X^{f_2}$ and $C_Y^{f_2}$ are
disjoint. Therefore it must be that the range of $g$ restricted to $\{n: 2n
\in C_X^{f_2}\}$ intersects all the $X$-classes of $\{n: 2n \in C_X^{f_2}\}$,
and thus the range of $f_2\restriction C_X^{f_2}$ intersects all the $X\oplus
Y$-classes of $C_X^{f_2}$. Thus there cannot be any odd $2m+1$ such that $f_2
(2m+1)\in C_X^{f_2}$. Otherwise let $f_2 (2m+1)=2n \in C_X^{f_2}$, and let
$2n'\in C_X^{f_2}$ be such that $f_2(2n') \rel{X\oplus Y} 2n$; as $f_2 (2m+1)
\rel{X\oplus Y} f_2(2n')$, it would follow $2m+1 \rel{X\oplus Y} 2n'$, a
contradiction.
\end{proof}

\begin{lemma}\label{ontoY}
For every $m$, either there is an odd $y$ so that $2m+1 \rel{X\oplus Y}
f_2(y)$ or there is an odd $y$ so that $2m+1 \rel{X\oplus Y} f_2^{(2)}(y)$.
\end{lemma}
\begin{proof}
Define the function
\[
g(n)=
\begin{cases}
\dfrac{f_2(2n+1)-1}{2}, &\textrm{if $f_2(2n+1)$ is odd},\\
\dfrac{f_2^{(2)}(2n+1)-1}{2}, &\textrm{otherwise}.
\end{cases}
\]
We first observe that $g(n)$ is always an integer. If $f_2(2n+1)$ is even,
then $f_2(2n+1)$ must be in $C^{f_2}_Y$, since $f_2^{-1}(C_X^{f_2})$ cannot
contain $2n+1$ by the previous Lemma. Thus
$f_2^{(2)}(2n+1)$ is odd. Suppose that $n\rel{Y}m$. Since $f_2$ is a
reduction of $X\oplus Y$ to $X\oplus Y$, we see that $f_2(2n+1)\rel{X\oplus
Y}f_2(2m+1)$, thus they have the same parity. Thus either
$g(n)=\frac{f_2(2n+1)-1}{2}\rel{Y} \frac{f_2(2m+1)-1}{2}=g(m)$ or
$g(n)=\frac{f_2^{(2)}(2n+1)-1}{2}\rel{Y} \frac{f_2^{(2)}(2m+1)-1}{2}=g(m)$.
Now, suppose that $g(n)\rel{Y}g(m)$. If $f_2(2n+1)$ has the same parity as
$f_2(2m+1)$, then we can use the fact that $f_2$ is a reduction of $X\oplus
Y$ to $X\oplus Y$ to conclude that $n\rel{Y} m$. If they have different
parities, then we have without loss of generality: $g(n)=
\frac{f_2(2n+1)-1}{2} \rel{Y}\frac{f_2^{(2)}(2m+1)-1}{2}=g(m)$. But then we
have $f_2(2n+1)\rel{X\oplus Y} f_2^{(2)}(2m+1)$. But since $f_2$ is a
reduction of $X\oplus Y$ to $X\oplus Y$, we have $2n+1\rel{X\oplus
Y}f_2(2m+1)$, but the latter is even. This is a contradiction. Thus we have
shown that $g$ is a reduction of $Y$ to $Y$. Since $Y$ is self-full, we
conclude that every class of $Y$ intersects the range of $g$. That is, $m
\rel{Y} g(k)$ for some $k$. If $f_2(2k+1)$ is odd, then $2m+1\rel{X\oplus Y}
f_2(2k+1)$. If $f_2(2k+1)$ is even, then $2m+1\rel{X\oplus Y}
f_2^{(2)}(2k+1)$.
\end{proof}

\begin{lemma}
For every $k$, $[k]_{X\oplus Y}$ intersects the range of $f_2$.
\end{lemma}
\begin{proof}
If $k=2m+1$ then Lemma~\ref{ontoY} ensures that $[k]_{X\oplus Y}$ intersects the
range of $f_2$.

Consider now $k=2m$.  If $2m\in C_X^{f_2}$ then we have shown in Lemma
\ref{Cx} that $[k]_{X\oplus Y}$ intersects the range of $f_{2}$. So we
suppose that $f_2(2m)=2n+1$. But we have shown in  Lemma~\ref{ontoY} that
$2n+1$ is $X\oplus Y$-equivalent to $f_2(y)$ for an odd $y$ or $f_2^{(2)}(y)$
for an odd $y$. In the former case, we have $f_2(2m)\rel{X\oplus Y} f_2(y)$,
which would imply that $2m\rel{X\oplus Y} y$ which is impossible since $y$ is
odd. Thus we have $f_2(2m)\rel{X\oplus Y} f_2^{(2)}(y)$, thus $2m\rel{X\oplus
Y} f_2(y)$.
\end{proof}

This shows that the range of $f_2$ intersects every $X\oplus Y$-class, so the range of $f$ must also intersect every $X\oplus Y$-class.
Thus $X\oplus Y$ is self-full.
\end{proof}

\begin{lemma}\label{darksAreCCR}
Every dark ceer $X$ is \CCR.
\end{lemma}

\begin{proof}
Suppose $C$ is a $\Pi^0_1$-equivalence relation and has infinitely many
classes. Then $\Id\leq C$, because $C$ is $\Pi^0_1$. Indeed, if $C=W_{e}^{c}$
and $W_{e}$ is the $e$-th c.e. set, then a reduction $f: \Id \rightarrow C$
can be defined as follows: to define $f(n)$ search for the least $\langle x,
s\rangle$ so that $\{\langle f(i), x \rangle: i<n\} \subseteq W_{e,s}$ (such
a number exists since $C$ has infinitely many classes) and let $f(n)=x$.
It follows that if $X$ is a ceer such that $\Id\not\leq X$, then there are $n,m$
so that $nXm$ and $n\cancel{\rel{C}}m$.
\end{proof}

\begin{corollary}\label{cor:hereditarily}
Every dark ceer is \HSF.
\end{corollary}

\begin{proof}
Every dark ceer is self-full by \cite{joinmeet} and is \CCR by Lemma
\ref{darksAreCCR}. By Theorem \ref{SFCCRsAreHSF}, dark ceers are \HSF.
\end{proof}

The next theorem shows that every non-universal degree in $\Ceers_{/\I}$ have infinitely many incomparable \HSF strong minimal covers. This is analogous to \cite[4.10]{joinmeet}, which proves the same result for the class of self-full ceers. Throughout the theorem and its proof, we employ the following notation. For every $k \ge 1$ we let $\Id_k$ denote a fixed ceer with exactly $k$ classes. Note that the exact choice does not matter since any two ceers with the same finite number of classes are equivalent. For uniformity of notation, for any ceer $A$, we allow $A\oplus \Id_0$ to represent $A$.

\begin{theorem}\label{thm:HSF-Istrongminimalcovers}
	Let $A$ be any non-universal ceer. Then there are infinitely many incomparable \HSF ceers $(E_l)_{l\in\omega}$ so that for every $n,l\in \omega$ and ceer $X$, $A\oplus \Id_n\leq E_l$ and
\[
X<E_l \Rightarrow (\exists k)[X\leq A\oplus \Id_k].
\]
\end{theorem}

\begin{proof}
We construct ceers $E_l$ with the property that the function $x \mapsto 2x$ is
a reduction $A \leq E_l$,
satisfying the following requirements:

\begin{itemize}

\item[$CCR^l_j$:] If $V_j$ is a co-c.e. equivalence relation with at least
    two distinct equivalence classes, then there are $x,y$ so that $x E_l y$
    and $x\cancel{\rel{V_j}}y$. (Here, $V_{j}=W_{j}^{c}$, i.e.\ the
    complement of the c.e. set $W_j$.)

\item[$SF_i^{k,l}$:] If $W_i$ intersects infinitely many $E_l$-classes which do not
    contain an even number, then $W_i$ intersects $[k]_X$.

\item[$D_j^{l,l'}$:] $\phi_j$ is not a reduction of $E_l$ to $E_{l'}$.

\end{itemize}
We first argue that any sequence of ceers $(E_l)_{l\in \omega}$ constructed with the above properties will be incomparable, \HSF and will have the properties that $A\oplus \Id_n\leq E_l$ and $X<E_l \Rightarrow (\exists k)[X\leq A\oplus \Id_k]$.

The $E_l$ are clearly incomparable by the $D$-requirements. By the fact that we will make $x\mapsto 2x$ be a reduction of $A$ to $E_l$, it is immediate that $A\leq E_l$ for each $l$. Thus if there is some pair $n,l$ so that $A\oplus \Id_n\not\leq E_l$, then $E_l\equiv A\oplus \Id_k$ for some $k\leq n$ by \cite[Lemma 2.5]{joinmeet}. Let $k$ be minimal so that there is some $E_l$ with $E_l\equiv A\oplus \Id_k$. Then $E_l\leq E_{l'}$ for every other $l'$, contradicting incomparability. So, $A\oplus \Id_n\leq E_l$ for every pair $n,l$.

Suppose that $\phi$ is a reduction witnessing $X\leq E_l$ for some ceer $X$. There are two cases to consider. Either $\phi$ intersects only finitely many, say $k$, classes which contain no even number, in which case $X\leq A\oplus \Id_k$, or by the $SF$-requirements applied to the c.e.\ image of $\phi$, $\phi$ is onto the classes of $E_l$. In the latter case, $X\equiv E_l$ \cite[Lemma 1.1]{joinmeet}. This shows that $X<E_l \Rightarrow (\exists k)[X\leq A\oplus \Id_k]$. Applying the same argument to any reduction $\phi$ of $E_l$ to itself shows that $\phi$ must be onto the classes of $E_l$. Thus $E_l$ is self-full. Finally, by the CCR-requirements, $E_l$ is also \CCR. It follows by Theorem \ref{SFCCRsAreHSF} that the $E_l$ are each \HSF.

\subsubsection*{The strategies}	
We now give the strategies for each requirement:

To achieve that $f(x)=2x$ is a reduction of $A$ to $X$ we guarantee that
throughout the construction we $E_l$-collapse a pair of distinct even numbers $2x,2y$ if and only if $x\rel{A} y$.

For the $CCR^l_j$-requirement, we wait until we see a pair $\langle x,y
\rangle$ so that $\langle x,y \rangle$ appears in $W_j$. At this point we
pick a new odd $z$, we wait for either $\langle x, z\rangle$ or $\langle y,
z\rangle$ to appear in $W_j$. While waiting, restrain the equivalence class
of $z$ from being $E_l$-collapsed to other classes due to the action of
lower-priority requirements. If and when the wait is over, drop the
restraint. If at that point we see $\langle x, z\rangle \in W_j$ then we
$E_l$-collapse $x$ and $z$. If we see $\langle y, z\rangle \in W_j$ but not as
yet $\langle x, z\rangle \in W_j$, then we $E_l$-collapse $y$ and $z$. Note
that if $V_j$ is an equivalence relation with at least two equivalence
classes then eventually such a pair $\langle x, y\rangle$ appears in $W_j$.
Then for every $z$ either $x \cancel{V_j} z$ or $y \cancel{V_j} z$. So
whatever $z$ we pick, eventually we see $\langle x, z\rangle \in W_j$ or
$\langle y, z\rangle \in W_j$.

For the $SF_i^{k,l}$-requirement, we wait until $W_i$ enumerates an odd number
$x$ which is not currently $E_l$-equivalent to any even number, nor lies in any
$E_l$-class currently restrained by a higher-priority requirement. We then
$E_l$-collapse this $x$ with $k$.

For the $D_j^{l,l'}$-requirement: We would like to employ the natural direct diagonalization strategy, i.e. take two new numbers in $E_l$ and choose to collapse them if and only if we do not choose to collapse their $\phi_j$-image in $E_{l'}$. The problem is that the $\phi_j$-images in $E_{l'}$ may be $E_{l'}$-equivalent to even numbers (in addition to the usual problem of them being equivalent to the finitely many numbers restrained by higher-priority strategies). So, we cannot control whether or not they collapse, because this is determined by $A$. The solution is to use the non-universality of $A$. In particular, as long as it appears that $\phi_j$ is giving a reduction of $E_l$ into the classes of even numbers in $E_{l'}$ (i.e. into $A$), we will encode a universal ceer into $E_l$. This will have one of two possible outcomes: Either $\phi_j$ will eventually have elements in its range which are not $E_{l'}$-equivalent to even numbers, or we will permanently witness a diagonalization. In either case, our threat to encode a universal ceer into $E_l$ will have a finite outcome, and we will not actually make $E_l$ universal.

More explicitly, the $D_j^{l,l'}$-strategy is as follows: We fix a universal ceer $T$. As we proceed, we will choose a sequence of parameters $a_0,a_1,\ldots$ which are not $E_l$-equivalent to any even numbers, and we will ensure $a_i \rel{E_{l,s}} a_k$ if and only if $i\rel{T_s}k$. At every stage, we will have only chosen finitely many parameters, thus will have encoded only a finite fragment of $T_s$. When we see that for every pair $i,k$ for which we have chosen parameters, $a_i \rel{E_l} a_k$ if and only if $\phi_j(a_i) \rel{E_{l'}} \phi_j(a_k)$, we will choose a new parameter $a_m$ for the least $m$ where we had not yet chosen $a_m$. In particular, we take one more step towards making $i\mapsto a_i$ be a reduction of $T$ to $E_{l'}$. We stop this strategy if we see two numbers $i,k$ so that $\phi_j(a_i),\phi_j(a_k)$ are neither $E_{l'}$-equivalent to each other nor to an even number nor a number restrained for a higher priority requirement. If this happens, we are ready to perform a direct diagonalization. We $E_l$-collapse $a_i$ and $a_k$, and we restrain the pair $\phi_j(a_i),\phi_j(a_k)$ from any future collapse in $E_{l'}$. This strategy either ends with the direct diagonalization as described, or by $\phi_j$ being partial, or by there being a pair $i,k$ so that $a_i \rel{E_{l}}a_k$ if and only if $\phi_j(a_i) \cancel{\rel{E_{l'}}} \phi_j(a_k)$. In any case, $\phi_j$ will not be a reduction of $E_l$ to $E_{l'}$.
	
Building the sequence $(E_l)_{l\in \omega}$ is carried out via a standard finite injury priority construction.

\subsection*{The construction:}
We build $(E_l)_{l\in \omega}$ in stages by using the collapsing technique as described at the
end of Section~\ref{sct:introduction}.

To \emph{initialize} a $D$-requirement or a $CCR$-requirement at a stage
means to cancel its parameters, if any, and to drop its restraint, if any. Further, a $D$-requirement which was satisfied is no longer considered satisfied after initialization. To
initialize an $SF$-requirement means to do nothing.

\subsubsection*{Stage $0$}  Let $X_{0}=\Id$.
All requirements are initialized.

\subsubsection*{Stage $s+1$}
We begin the stage by collapsing even numbers in every $E_l$ so that $2x \rel{E_{l,s}} 2y$ whenever $x \rel{A_s} y$.
	
If $R=CCR_j^l$, we say that $R$ is \emph{already satisfied at $s+1$} if there is
$\langle
w,z\rangle \in W_{j,s}$ so that $w \rel{E_{l,s}} z$.
If $R=SF_i^{k,l}$, we say that $R$ is \emph{already satisfied at $s+1$} if there is
$x \in W_{i,s}$ so that $x \rel{E_{l,s}} k$.

A \emph{new} number $x$ at stage $s+1$ is an odd  number bigger than any
number so far mentioned in the construction. In particular, if $x$ is new at stage $s+1$, then $[x]_{E_{l,s}}=\{x\}$ for each $l$.

Scan the requirements $R$ in decreasing order of priority. Suppose you have
scanned all $R'<R$, and distinguish the following cases:

\subsubsection*{$R=CCR_{j}^l$}
If $R$ is already satisfied then it drops its restraint if any, and we move
to the next requirement.

If the requirement is not already satisfied at $s+1$ and it is initialized
(i.e. has no chosen parameters), but we now see numbers $x,y$ with $\langle
x,y \rangle \in W_{j,s}$, then take a new odd number $z$ which becomes the
\emph{current witness} of $R$. \emph{Restrain} the equivalence class of $z$
from being $E_l$-collapsed to other classes due to  lower-priority
requirements.
After this, if we see $\langle x, z\rangle \in W_{j,s}$ or $\langle y,
z\rangle \in W_{j,s}$, then \emph{drop} the restraint, \emph{$E_l$-collapse}
$x$ and $z$, or $y$ and $z$ according to the case. Whether $CCR_{j}^l$ chooses
$z$, or we cause $E_l$-collapse, we say that $CCR_{j}^l$ has \emph{acted at stage $s+1$}.
After this action, \emph{stop the stage} and \emph{initialize} all lower-priority requirements. If we neither choose $z$ nor cause $E_l$-collapse, then proceed to the next requirement.

\subsubsection*{$R=SF_{i}^{k,l}$}
Suppose the requirement is not already satisfied at $s+1$, but we now see an odd
number $x \in W_{i,s}$ not in the equivalence class of any number currently
restrained by some higher priority $R'$, nor $E_{l,s}$-equivalent to any even
number. We view this $x$ as the \emph{once-and-for-all witness} of the
requirement. Then \emph{$E_l$-collapse} $x$ with $k$. We say that $SF_{i}^{k,l}$ has
\emph{acted at stage $s+1$}. After the action \emph{stop the stage} and
\emph{initialize} all requirements $R'$ having lower-priority than $R$. If no such $x$ is found, then proceed to the next requirement.

\subsubsection*{$R=D_{j}^{l,l'}$}
If $R$ is currently initialized and not yet satisfied, then choose two new odd numbers $a_0,a_1$ which
become the \emph{current witnesses} of $R$. \emph{Restrain} the classes of
the two witnesses from $E_l$-collapses due to lower-priority requirements. After choosing $a_0$ and $a_1$ and placing restraint, \emph{stop the stage} and
\emph{initialize} all requirements $R'$ having lower-priority than $R$.

If $R$ already has witnesses, let $m$ be least so that $a_m$ is not yet chosen. For $i,k<m$, $E_l$-collapse $a_i$ with $a_k$ if and only if $i\rel{T_s} k$ (recall that $T$ is our fixed universal ceer). If this causes any new collapses, then we say we have caused $E_l$-collapse. We distinguish the following 3 cases:

\emph{Case 1:} There are $i,k<m$ so that $\phi_j(a_i),\phi_j(a_k)$ both converge by stage $s$ to numbers which are not $E_{l',s}$-equivalent to an even number nor a number restrained by higher-priority requirements and $\phi_j(a_i)\cancel{\rel{E_{l',s}}}\phi_j(a_k)$. We $E_l$-collapse $a_i$ with $a_k$ and place a permanent restraint to not allow lower-priority requirements to cause $\phi_j(a_i)$ to $E_{l'}$ collapse with $\phi_j(a_k)$. We let $\phi_j(a_i)$ and $\phi_j(a_k)$ permanently be parameters $c$ and $d$ of the requirement. We now say that the requirement is satisfied (this is permanent unless the requirement is reinitialized). After this action, \emph{stop the stage} and \emph{initialize} all lower-priority requirements.

\emph{Case 2:} Not Case 1 and $\phi_j(a_i)$ converges by stage $s$ for every $i<m$ and for each $i,k<m$, $a_i \rel{E_{l,s}} a_k$ if and only if $\phi_j(a_i)\rel{E_{l,s}}\phi_j(a_k)$. In this case, we define $a_m$ to be a new odd number. After this action, \emph{stop the stage} and \emph{initialize} all lower-priority requirements.

\emph{Case 3:} Not Case 1 nor Case 2: If we have caused $E_l$-collapse, then \emph{stop the stage} and \emph{initialize} all lower-priority requirements. Otherwise, move to the next requirement.

Notice that eventually we are sure to move on to stage $s+2$ since the action of
any initialized $D$-requirement stops the stage, and there are cofinitely
many initialized requirements at any stage.

\subsection*{The verification:}
We split the verification into the following lemmata. Notice that at each
stage, exactly one requirement acts.

Let us say that a witness for a requirement $R$ is \emph{$E_l$-active} if it is a parameter $z$ for a $CCR_j^{l}$-requirement or is an $a_i$ parameter for a $D_j^{l,l'}$-requirement or is a $c$ or $d$ parameter for a $D_j^{l',l}$-requirement. That is, it is a parameter for which some requirement restrains $E_l$-collapse.

\begin{lemma}\label{lem:hs2}

At every stage $s$, if $x$ and $y$ are $E_l$-active witnesses for different requirements, then $x\cancel{\rel{E_l}} y$.
	
At every stage $s$, if $x,y$ are $E_l$-active witnesses and $x\rel{E_{l,s}} y$, then $x=a_i$ and $y=a_j$ for a $D$-requirement and $i\rel{T_s}j$.

At every stage $s$, if $x$ is $E_l$-active, then $x$ is not $E_{l,s}$-equivalent to any even number.
\end{lemma}

\begin{proof}
We prove all three claims by simultaneous induction on $s$. The lemma is certainly true at stage $0$, as there are no $E_l$-active witnesses. If the claim is true at $s$, then
at $s+1$ at most one $R$ can effect $E_l$. Either $R$ $E_{l,s+1}$-collapses the class of
its witness with some other class, or (if $R$ is a $D$-requirement)
$E_{l,s+1}$-collapses classes of its witnesses, or $R$ appoints new $E_l$-active
witnesses.

If $R$ is a $CCR$-requirement, then $R$'s witness $z$ ceases to be active. So, it has collapsed $[z]_{E_{l,s}}$ to some other class $[w]_{E_{l,s}}$. By inductive hypothesis, the only member of $[z]_{E_{l,s}}$ which was either $E_l$-active or even was $z$ itself. Since $z$ is no longer active at stage $s+1$, we have added neither an $E_l$-active number nor an even number to $[w]_{E_{l,s+1}}$.

If $R$ is a $D$-requirement, it collapses $a_i$ with $a_j$ only if $i\rel{T_s} j$. Since neither $a_i$ nor $a_j$ were $E_{l,s}$-equivalent to another requirement's $E_l$-active witness or an even number by inductive hypothesis, the combined class maintains this property.

If $R$ is an $SF_i^{k,l}$-requirement, then it collapses its once and for all witness $x$ to $k$. This $x$ is not $E_{l,s}$-equivalent to any $E_{l,s+1}$-active number since it is not $E_{l,s}$-equivalent to a number restrained by a higher-priority requirement (and this action initializes all lower-priority requirements) nor an even number, and it initializes all lower-priority requirements. Thus it has added neither an even number nor an $E_l$-active number to $[k]_{E_{l,s+1}}$.

Finally, if $R$ appoints new $E_l$-active witnesses, these witnesses are either new, so their classes are currently disjoint
singletons or are the $c$ and $d$ for a $D$-requirement, which are appointed only if they are in distinct $E_l$-classes from every other $E_l$-active or even number.
\end{proof}

\begin{lemma}\label{lem:hs4}
	The map $i\mapsto 2i$ is a reduction witnessing $A \leq E_l$ for each $l$.
\end{lemma}
\begin{proof}
	Since no active parameter is ever equivalent to an even number by Lemma \ref{lem:hs2}, we never cause collapse of two even numbers aside from at the beginning of each stage. At that point, we collapse $2x \rel{E_{l,s}} 2y$ if and only if $x \rel{A_s} y$. Since we never cause any other collapse among even numbers, we have $2x \rel{E_l} 2y$ if and only if $x\rel{A} y$.
\end{proof}

\begin{lemma}\label{lem:hs1}
	Every requirement is initialized finitely often, eventually stops acting, and
	sets up only a finite restraint.
\end{lemma}

\begin{proof}
	The proof is by induction on the priority ranking of the requirement $R$,
	since every requirement acts only finitely often after each initialization. This is clear for each requirement except $D$-requirements. Clearly they can act via Case 1 at most once after last initialization. Suppose towards a contradiction that it acts infinitely often via Case 2 after last initialization. Then $\psi:i\mapsto a_i$ is a reduction of $T$ to $E_{l'}$. Further the range of $\psi$ is contained in the classes which contain even numbers and the finitely many (by inductive hypothesis) classes restrained for higher priority strategies. Thus, using the fact that $E_{l'}$ restricted to the even numbers is equivalent to $A$ by Lemma \ref{lem:hs4}, we can alter $\psi$ to give a reduction of $T$ to $A\oplus \Id_k$ for some $k$. But since $A$ is assumed to be non-universal, and the universal degree is uniform join irreducible \cite[Prop. 2.6]{UniversalCeers}, $A\oplus \Id_k$ is non-universal. This is a contradiction. Finally, since the requirement acts only finitely often in Case 2, it defines only finitely many parameters $a_i$. Thus, it can only cause $E_l$-collapse finitely often, and can only initialize lower priority requirements in Case 3 finitely often. So, $D$-requirements also act finitely often,  place finite restraint, and initialize lower-priority requirements finitely often after their last initialization.
\end{proof}

\begin{lemma}\label{lem:hs3}
Every requirement is eventually satisfied.
\end{lemma}

\begin{proof}
Notice that $CCR$- or $SF$-requirements can never be injured after they have
acted. If they collapse, then the collapse permanently satisfies the requirement. Suppose
that every $R'<R$ is eventually satisfied.

We now distinguish the various possibilities for $R$.

\subsubsection*{$R=SF_{i}^{k,l}$}
We must only worry if $W_{i}$ intersects infinitely many $E_l$-classes
not containing even numbers. In this case, eventually some odd number $x$
appears in $W_{i}$ not as yet $E_l$-equivalent to any even number and avoiding
the finite restraint imposed by the higher-priority requirements. So $R$
acts by picking such an $x$ (the once-and-for-all witness of the
requirement), and $x$ is $E_l$-collapsed to $k$, which makes $R$ permanently
satisfied.

\subsubsection*{$R=CCR_{j}^l$}
By the fact that lower-priority $SF$-requirements will choose their
once-and-for-all witnesses avoiding higher-priority restraints, and thus not interfering with $CCR_j^l$, the
only possible injury to $CCR_j^l$ can be made, thanks to Lemma~\ref{lem:hs2},
by higher-priority requirements. By Lemma~\ref{lem:hs1} let $s_{0}$ be
the least stage such that these higher-priority requirements do not act at
any $s \ge s_0$. If $V_{j}$ is an equivalence relation with at least two
distinct equivalence classes, then eventually we either see at some point
that $R$ is already satisfied without having acted (and so the requirement is
permanently satisfied), or we see at some point a pair $\langle x, y \rangle$
to appear in $W_{j}$. At that point, we pick the final witness $z$, a new odd
number. At some later stage we collapse $z$ with $x$ or $y$, which makes $R$
permanently satisfied.

\subsubsection*{$R=D_{j}^{l,l'}$}
By the fact that lower-priority $SF$-requirements will choose their
once-and-for-all witnesses avoiding higher-priority restraint, and thus not interfering with $D_i$, the
only possible injury to $D_i$ can be made, thanks to Lemma~\ref{lem:hs2}, by
higher-priority requirements. By Lemma~\ref{lem:hs1} let $s_{0}$
be the least stage such that these higher-priority requirements do not act at any
$s \ge s_0$. Requirement $D_j^{l,l'}$ is never re-initialized after $s_{0}$. After stage $s_0$, we have several possible outcomes for $D_j^{l,l'}$. We have already shown that it cannot take Case 2 infinitely often. If it ever takes Case 1, then we have two numbers $a_i,a_k$ so that $a_i \rel{E_l} a_k$ and $\phi_j(a_i)\cancel{\rel{E_{l'}}}\phi_j(a_k)$. The remaining possibility is that it takes Case 3 co-finitely often. In this case, we see that either $\phi_j$ is not total or is not a reduction (either case is witnessed on the finite set $\{a_i\mid i<m\}$).
\end{proof}

The proof of the theorem is now complete.
\end{proof}

Note the following corollary which follows from letting $A$ be $\Id$:

\begin{corollary}\label{thm:light-hsf}
	There are light ceers which are \HSF.
\end{corollary}

We conclude by showing that the property of being hereditarily self-full is
itself hereditary in one sense and not in another.

\begin{observation}\label{obs:not-hereditarily}
If $Z$ is \HSF, then there is a self-full ceer $X$ so that $Z\oplus X$ is not
\HSF.
\end{observation}
\begin{proof}
Let $Z$ be \HSF, and let $X$ and $Y$ be self-full so that $X\oplus Y$ is
non-self-full. Then $Z\oplus X$ is not \HSF because $Z\oplus X\oplus Y$ is
non-self-full.
\end{proof}

\begin{observation}\label{HSFclosed}
If $X$ is \HSF and $Y$ is \HSF, then $X\oplus Y$ is \HSF.
\end{observation}

\begin{proof}
Let $E$ be any self-full ceer. Then since $Y$ is \HSF, we see that $Y\oplus
E$ is self-full. Thus, since $X$ is \HSF, we see that $X\oplus (Y\oplus E)$
is self-full. We conclude that whenever $E$ is self-full, $(X\oplus Y)\oplus
E$ is self-full, showing that $X\oplus Y$ is \HSF.
\end{proof}

%\bibliographystyle{plain}
%\bibliography{UniformJoinOfSF}

\end{document}